\documentclass[12pt,14paper]{amsart}
\usepackage{a4,amsmath,amssymb,epic}
\usepackage{graphicx}
\usepackage{multicol}
\usepackage{rotating}
\usepackage[margin=3.4cm]{geometry}
\usepackage{amssymb,amsthm}
\usepackage{enumitem}
  \usepackage[all]{xy}

\allowdisplaybreaks
\numberwithin{equation}{subsection}
\theoremstyle{plain}
\newtheorem{lem}{Lemma}

\newtheorem{thm}[lem]{Theorem}
\newtheorem{cor}[lem]{Corollary}

\newtheorem*{lem*}{Lemma}
\newtheorem*{Acknowledgements*}{Acknowledgements}
\newtheorem*{prop*}{Proposition}
\newtheorem*{thm*}{Theorem}
\newtheorem*{cor*}{Corollary}
\newtheorem*{conj*}{Conjecture}
\theoremstyle{remark}

\newtheorem*{defn}{Definition}

\newcommand{\rad}{\operatorname{rad}}

\newcommand{\Id}{\mathrm{Id}}

\begin{document}

\title{Brauer algebras of Type $C$ are Cellularly Stratified Algebras}
\author{C.~Bowman}
\address{Corpus Christi College, Cambridge, CB2 1RH, England, UK}
\date{26th January 2011}

\begin{abstract} 
In a recent paper Cohen, Liu and Yu introduce the Type $C$ Brauer algebra.  We show that this algebra is an iterated inflation of hyperoctahedral groups, and that it is cellularly stratified.  This gives an indexing set of the standard modules, results on decomposition numbers, and the conditions under which the algebra is quasi-hereditary.
\end{abstract}

\maketitle
\section{Introduction}
In a recent paper Cohen, Liu and Yu introduce the Type $C$ Brauer algebra. It is defined as the fixed points of a diagram automorphism of the Type $A$ Brauer algebra.  They show that this algebra is cellular.  

K\"onig and Xi introduced in \cite{cellular} the concept of an iterated inflation and showed this to be equivalent to being cellular.  This construction allows one to study a larger cellular algebra in terms of the layers of the inflation.  This framework allows us to construct the standard modules.

In \cite{damn} Hartmann, Henke, K\"onig and Paget define a cellularly stratified algebra, as an iterated inflation with certain conditions.  The extra structure allows for the comparison of certain decomposition numbers with those of the input algebras; 
it also gives us natural definitions of Young modules and associated Schur algebras.

We shall follow the model laid out in the papers of K\"onig and Xi to show that the Brauer algebra of type $C$ is an iterated inflation of hyperoctahedral groups.  On the way we shall define the idempotents necessary to prove that the algebra satisfies the extra conditions to be cellularly stratified.  We then discuss the immediate results on standard modules and decomposition numbers, as well as deriving the conditions under which the algebra is quasi-hereditary.

\section{Definitions and Examples}
The Dynkin diagram of type $C_n$ arises from that of type $A_{2n+1}$ as the fixed points of a diagram automorphism.  This automorphism, $\sigma$, can be described geometrically as reflection in the vertical axis of the diagram.  Abusing notation we shall also let $\sigma$ denote the induced map on the Coxeter groups, Brauer algebras, and planar diagrams.

\subsection{The Type $C$ Brauer algebra}

We give the definition of the Type $C$ Brauer algebra and a diagrammatic visualisation.  Let $R$ be a commutative ring with invertible element $\delta$.  For $n \in \mathbb{N}$, the Brauer algebra of Type $C_n$ over $R$ with loop parameter $\delta$, denoted by $B(C_n, R, \delta)$ is the $R$-algebra generated by $r_1, \ldots , r_{n}$ and $e_1, \ldots, e_{n}$ subject to the following relations:
\begin{align*}
r_2r_1r_2r_1 &=r_1r_2r_1r_2 , \\
r_i^2 &=1	 , \\
r_i r_{i+1} r_i &= r_{i+1} r_i r_{i+1} \text{ for } i\neq 1 	 , \\
r_ir_{i+1} &= r_{i+1} r_i   , \\
e_i^2 &= \delta^2 e_i \text{ for } i \neq 1 , \\
e_1^2 &= \delta e_1	, \\
e_i e_{i+1} &= e_{i+1} e_i  , \\
r_ie_i&=e_ir_i = e_i	, \\
e_ir_{i\pm 1}&= r_{i\pm 1}e_i \text{ for }i > 2, \\
r_{i\pm 1}r_ie_{i\pm 1}&=e_ie_{i\pm 1}\text{ for } i > 2, \\
r_ie_jr_i &= r_je_ir_j \text{ for } i,j > 1	, \\
r_2r_1e_2&=r_1e_2 , \\
r_2e_1r_2e_1 &=e_1e_2e_1	, \\
(r_2r_1r_2)e_1 &= e_1(r_2r_1r_2)	, \\
e_2r_1e_2 &= \delta e_2	, \\
e_2 e_1 e_2 &=\delta e_2 , \\
e_2r_1r_2 &= e_2r_1 , \\
e_2e_1r_2 &=e_2e_1.
\end{align*}
where the elements can be seen diagrammatically as follows:
\begin{align*}
r_1 = \begin{minipage}{54mm}
\def\objectstyle{\scriptstyle}
\xymatrix@=2pt{
		&&			&&			&\ar@{--}[dddd]\\
\circ  \ar@{-}[dd]		&\cdots&\circ  \ar@{-}[dd]		&&\circ\ar@{-}[ddrr] 		&&\circ 	\ar@{-}[ddll]	&&\circ 	 \ar@{-}[dd]	&\cdots&\circ  \ar@{-}[dd]	\\	
&&&&\\	
\circ 		&\cdots&\circ 		&&\circ 		&&\circ 		&&\circ 		&\cdots&\circ 	\\
		&&			&&			&
}
\end{minipage}
\ \ \ \ \ \ r_i = \begin{minipage}{54mm}
\def\objectstyle{\scriptstyle}
\xymatrix@=2pt{
		&&			&&			&&&\ar@{--}[dddd]\\
\circ  \ar@{-}[dd]		&\cdots&\circ\ar@{-}[ddrr]   	&&\circ\ar@{-}[ddll] 	& \cdots	&\circ\ar@{-}[dd]	&&\circ  \ar@{-}[dd]		&\cdots&\circ\ar@{-}[ddrr]   	&&\circ\ar@{-}[ddll] 	& \cdots	&\circ\ar@{-}[dd]\\
&&&&\\	
\circ 		&\cdots&\circ 		&&\circ 	&\cdots	&\circ 		&&\circ 		&\cdots&\circ 		&&\circ 	&\cdots	&\circ \\
		&&			&&	&&		&
}
\end{minipage}\\
e_1 = \begin{minipage}{54mm}
\def\objectstyle{\scriptstyle}
\xymatrix@=2pt{
		&&			&&			&\ar@{--}[dddd]\\
\circ  \ar@{-}[dd]		&\cdots&\circ  \ar@{-}[dd]		&&\circ\ar@{-}[rr] 		&&\circ 	 	&&\circ 	 \ar@{-}[dd]	&\cdots&\circ  \ar@{-}[dd]	\\	
&&&&\\	
\circ 		&\cdots&\circ 		&&\circ 	\ar@{-}[rr]	&&\circ 		&&\circ 		&\cdots&\circ 	\\
		&&			&&			&
}
\end{minipage}
\ \ \ \ \ \ e_i = \begin{minipage}{54mm}
\def\objectstyle{\scriptstyle}
\xymatrix@=2pt{
		&&			&&			&&&\ar@{--}[dddd]\\
\circ  \ar@{-}[dd]		&\cdots&\circ  	&&\circ\ar@{-}[ll] 	& \cdots	&\circ\ar@{-}[dd]	&&\circ  \ar@{-}[dd]		&\cdots&\circ\ar@{-}[rr]   	&&\circ 	& \cdots	&\circ\ar@{-}[dd]\\
&&&&\\	
\circ 		&\cdots&\circ 		&&\circ \ar@{-}[ll] 	&\cdots	&\circ 		&&\circ 		&\cdots&\circ 		&&\circ\ar@{-}[ll] 	&\cdots	&\circ \\
		&&			&&	&&		&
}
\end{minipage}
\end{align*}
for $i\geq2$.  This algebra is shown in \cite{cohen} to be cellular with respect to the involution given by reflection in the horizontal axis.

\subsection{Hyperoctahedral Groups}
The $n$-dimensional hyperoctahedral group $H_n$, is the Weyl group of type $C_n$, and as a wreath product is $\Sigma_2\wr\Sigma_n$ (where $\Sigma_n$ is the symmetric group on $n$ letters).  We shall view this group as a subgroup of $\Sigma_{2n}$ with generators
\begin{align*}
r_1&= (1, n+1)(2, n+2) \ldots (n, 2n),		\\
r_i&= ( i-1, i) \text{ for } 2\leq i \leq n.
\end{align*}
The Weyl group of type $C_n$ is well-known to be the fixed points of the Weyl group of type $A_{2n+1}$ under the graph automorphism $\sigma$.  If we picture the group as the symmetric planar diagrams with no horizontal arcs, then it is cellular with respect to the involution given by reflection in the horizontal axis (or equivalently with respect to the inverse map).

\section{Cellularly Stratified Algebras}

\subsection{Cellular Algebras}
The original definition of a cellular algebra, given by Graham and Lehrer in \cite{GL} has been shown to be equivalent to the following definition due to K\"onig and Xi.
\begin{defn}
Let $A$ be a $k$-algebra.  Assume there is an anti-automomorphism $i$ on $A$ with $i^2= \text{id}$.  A two-sided ideal $J$ in $A$ is called a \emph{cell ideal} if and only if $i(J)=J$ and there exists a left ideal $\Delta \subset J$ such that $\Delta$ has finite $k$-dimension and that there is an isomorphism of $A$-bimodules $\alpha:J \cong \Delta \otimes_k i(\Delta)$ making the following diagram commutative:
\[\xymatrix@=6pt{
&J \ar@{>}[rrrr]^{\alpha}\ar@{>}[dddd]^{i}	&&&&\Delta \otimes_k i(\Delta)\ar@{>}[dddd]^{x \otimes y \mapsto i(y) \otimes i(x)}\\
&\\
&\\
&\\
&J \ar@{>}[rrrr]^{\alpha}			&&&&\Delta \otimes_k i(\Delta)	
}\]
The algebra $A$ is called \emph{cellular} if and only if there is a vector space decomposition $A=J'_1 \oplus \ldots \oplus J'_n$ with $i(J'_j)=J'_j$ for each $j$ and such that setting $J_j=\oplus_{k=1}^jJ'_l$ gives a chain of two-sided ideals of $A$ and for each $j$ the quotient $J'_j=J_j/J_{j-1}$ is a cell ideal of $A/J_{j-1}$.
\end{defn}

\subsection{Iterated Inflations}

It is proven in \cite{cellular} that any cellular algebra can be exhibited as an iterated inflation of copies of the ground field $k$.  Conversely, any iterated inflation of cellular algebras is again cellular.  We shall review the definition of an inflation, but shall avoid the technical details of the definition of an iterated inflation, and instead refer to \cite{cellular}.  

Given a $k$-algebra $B$, a $k$-vector space $V$, and a bilinear form $\varphi:V \times V \to B$ we define an associative algebra (possibly without unit) as follows: as a $k$-module $A= V \otimes V \otimes B$, the multiplication is defined on basis elements:
\begin{align*}
(a \otimes b \otimes x)(c \otimes d \otimes y) = (a \otimes d \otimes x \varphi(b,c) y).
\end{align*}
We are interested in the case where $B$ is a cellular algebra, and therefore comes equipped with an anti-automorphism $\sigma$ on $B$.  We then have that the form $\varphi$ respects the anti-automorphism, so that $\sigma\varphi(v,w)=\varphi(w,v)$.  We can then define an anti-automorphism, $i$, on $A$ by letting $i(a \otimes b \otimes x)=b \otimes a \otimes \sigma(x)$.

This definition makes $A$ an associative algebra with an anti-automorphism.  We say that $A$ is an \emph{inflation} of $B$ along $V$.  $A$ will not usually have a unit element, but may have idempotents.

Let $B$ be an inflated algebra (possibly without unit) and $C$ be an algebra (with a unit).  We define an algebra structure on the vector space $A=B\oplus C$ which extends the two given structures in a way so that $B$ is a two-sided ideal and $A/B=C$.  We require that $B$ is an ideal, the multiplication is associative, and that there exists a unit element of $A$ which maps onto the unit of the quotient $C$.  The necessary conditions are outlined in \cite{cellular}.

An inductive application of this procedure to algebras $C, B_1, B_2 \ldots $ ensures that the inflation pieces, $B_i=V_i \otimes V_i\otimes B_i'$, are subquotients of ideals in the algebra $A$.

Let $A$ be a cellular algebra with anti-involution $i$, and inflation decomposition:
\begin{align*}
A = \bigoplus^m_{j=1} V_j \otimes V_j \otimes B_j,
\end{align*}
where $V_j$ is a vector space and $B_j$ is a cellular algebra with respect to an anti-involution $\sigma_j$, such that the restriction of $i$ to $V_j \otimes V_j \otimes B_j$ is given by $w \otimes v \otimes b \mapsto v \otimes w \otimes \sigma_j(b)$. 

\subsection{Cellularly Stratified Algebras}

The following definition is lifted from \cite{damn}.

\begin{defn}
A finite dimensional associative algebra $A$ is called \emph{cellularly stratified} with stratification data $(B_1, V_1, \ldots , B_n,V_n)$ if and only if the following conditions are satisfied:
\begin{enumerate}
\item The algebra $A$ is an iterated inflation of cellular algebras $B_l$ along vector spaces $V_l$ for $l=1, \ldots, n$.
\item For each $l=1, \ldots, n$ there exist non-zero elements $u_l, v_l \in V_l$ such that
\begin{align*}
e_l = u_l \otimes v_l \otimes 1_{B_l},
\end{align*}
is an idempotent.
\item If $l > m$, then $e_le_m=e_m = e_m e_l$.
\end{enumerate}
\end{defn}

In \cite{damn} it is shown that the partition, BMW, and Brauer algebras are cellularly stratified. 

\section{The Brauer algebra of type $C$ as an iterated inflation}
In this section we prove that the Brauer algebra, $B(C_n,R, \delta)$ is an iterated inflation of the hyperoctahedral groups, providing $\delta \neq 0$.
\begin{defn}
An $(n,k)$-dangle is a partition of $\{1 , \ldots, n\}$ into $n-2k$ one-element subsets and $k$ two-element subsets, here $k \leq n/2$.
\end{defn}
We have that the graph automorphism, $\sigma$, naturally acts on the set of all $(n,k)$-dangles by permuting the nodes.  This can be seen as a reflection in the vertical axis of the diagram.
\begin{defn}
A symmetric $(n,k)$-dangle is an $(n,k)$-dangle which is invariant under the graph automorphism $\sigma$.
\end{defn}

We can geometrically represent a symmetric dangle, $d$, by a set of $n$ nodes labelled by the set $\{1, \ldots, n\}$, where there is an arc joining $i$ to $j$ if $\{i,j\} \in d$, and there is a vertical line starting from $i$ if $\{i\} \in d$.  For example

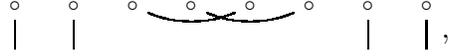
\begin{align*}
\begin{minipage}{54mm}
\def\objectstyle{\scriptstyle}
\xymatrix@=2pt{
\circ  \ar@{-}[dd]		&&\circ  \ar@{-}[dd]		&&\circ\ar@{-}@/_.5pc/[rrrr]	&&\circ 	\ar@{-}@/_.5pc/[rrrr]		&&\circ 		&&\circ	&&\circ \ar@{-}[dd] &&\circ \ar@{-}[dd]   \\
&&&&&&&&&&&\\	
&&&&&&&&&&&&&&\\
}
\end{minipage},
\end{align*}
is a symmetric $(8,2)$-dangle.  We let $V(n,k)$ denote the set of all symmetric $(n,k)$-dangles.  Dually, we let $V'(n,k)$ to be a copy of $V(n,k)$, but draw the vertical lines upwards, rather than downwards.  We shall often denote both spaces as $V(n,k)$ if 
we do not want to distinguish between the two spaces.

We now define the idempotents 
we shall need, let $f_k$ be the $\delta^{-k}$ times the diagram with $k$ arcs, each joining the nodes $n/2-i$ to $n/2+i$ for $i\leq k$, as below:
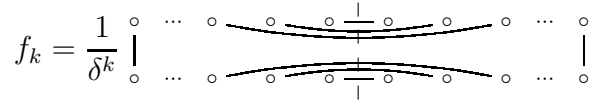
\begin{align*}
f_k=\frac{1}{\delta^k}\begin{minipage}{54mm}
\def\objectstyle{\scriptstyle}
\xymatrix@=2pt{
&&										&&&&&\ar@{--}[dddd]\\
\ar@{-}[dd]	\circ & \cdots & 	\circ &&\circ& &\circ 				& & 	\circ	\ar@{-}[ll]			&& \circ \ar@{-}@/^.3pc/[llllll]   	&& \circ\ar@{-}@/^.5pc/[llllllllll] &\cdots& \circ  \ar@{-}[dd]\\
&&&&&&&&&&&\\	
\circ&\cdots & 	 \circ	\ar@{-}@/^.5pc/[rrrrrrrrrr]  			&&  \circ \ar@{-}@/^.3pc/[rrrrrr]& & \circ  \ar@{-}[rr]		&&\circ 		&& \circ			&& \circ &\cdots&\circ 	&& 	\\
&&						&&&&				&&			&}
\end{minipage}  \end{align*}
Note that there idempotents are well-defined as we have assumed that $\delta \neq 0$.  We define a chain of two-sided ideals of $B=J_0 > J_1 > J_2 > \ldots $, where $J_i=Bf_iB$. By definition, $J_i = Bf_iB$ is the two-sided cell-ideal where all diagrams have at least $i$ horizontal arcs.

\begin{lem}\label{result}
Fix an index $k$ and let $B$ be the $k$-algebra $J_k/J_{k+1}$.  Then $B$ is isomorphic as a $k$-algebra to the inflation $V(n,k) \otimes V'(n,k) \otimes kH_{n-2k}$ of the hyperoctahedral group $H_{n-2k}$ along the free $k$-module $V(n,k)$.  In the proof we shall define the bilinear form.
\end{lem}

\begin{proof}
Suppose $d \in B(C_n, R, \delta)$ with $k$ vertical lines.  Then we get an element $e \in V(n,k)$ and an element of $f \in V'(n,k)$ by cutting all the vertical lines.  We also get a permutation $\pi \in \Sigma_{n-2k}$ in the following way: enumerate the top ends of the vertical lines in $d$ from left to right with the labels $\{1, \ldots, n-2k\}$, and the bottom row of nodes with the primed labels.  If $i$ and $j'$ are joined by a vertical line in $d$, then define $\pi(d)(i)=j$.  This determines an element $\pi(d)\in \Sigma_{n-2k}$.  As $d$ is fixed under the graph automorphism on the Brauer algebra, we get that $\pi_d$ is fixed under the graph automorphism $\sigma$ acting on the Weyl group; and therefore $\pi(d) \in H_{n-2k}$.  These datum $(d_1,d_2,\pi(d))\in V(n,k) \otimes V'(n,k) \otimes H_{n-2k}$ clearly determine $d$ uniquely.
We therefore define an isomorphism of $k$-modules:
\begin{align*}
\psi  	&: B \to V(n,k) \otimes V'(n,k) \otimes kH_{n-2k},	\\
	&: d \mapsto e \otimes f \otimes \pi_d.
\end{align*}
In order to define the multiplicative structure we need to define a bilinear form $\varphi : V(n,k) \otimes V'(n,k) \to H_{n-2k}$.  Let $x=e \otimes f \otimes \pi_{d_1}$, $y=g \otimes h \otimes \pi_{d_2}$ be two bases elements of $V(n,k) \otimes V'(n,k) \otimes kH_{n-2k}$.  The product $xy$ is of the form $e \otimes h \otimes \pi_{d_1} \varphi(f,g)\pi_{d_2}$.  We have two cases to consider.

\textbf{First Case:} Each element in $f$ has at least one common vertex with an element of $g$, or equivalently the number of arcs in $xy$ is equal to the number of arcs in $x$ or $y$, (whichever is the larger).  In this case we define $\varphi(f,g)=\delta^m \sigma$, which are defined as follows: the group element $\sigma$ is the permutation defined by the through strings of the product $\psi^{-1}(f \otimes f \otimes \Id) \cdot \psi^{-1}(g \otimes g \otimes \Id)$, and the integer $k$ is defined to be the number of closed loops in this product.

\textbf{Second Case:} There is at least one edge in $f$, which shares no vertex in $g$, or equivalently the number of arcs in $xy$ is strictly greater than the number of arcs in either $x$ or $y$.  In this case we define $\varphi(f,g)=0$.

We have defined the bilinear form so that it coincides exactly with the multiplication in a given layer of the iterated inflation, and so we are done.
\end{proof}
Now we must check that the layers fit together properly; if $d_1$, $d_2$ are from different layers (or equivalently have a different number of horizontal arcs) we want to discuss the product $d_1d_2$.
\begin{lem}
Let $d_1 \in J_l/J_{l+1}$ and $d_2 \in J_k/J_{k+1}$ be two diagrams in $A$, which map to $e \otimes f \otimes \pi_{d_1}$, and $g \otimes h \otimes \pi_{d_2}$ under the bilinear forms defined for their respective layers.  We assume $k\geq l$.  Then the product $d_1d_2$ is either an element of $J_k/J_{k-1}$, or is an element of $J_{k-1}$.  In the former case it corresponds under $\psi$ to a scalar multiple of an element $v \otimes g \otimes \sigma d_2$ where $v \in V(n,k)$ and $\sigma \in H_{n-2k}$.
\end{lem}
There is a similar statement for $k \leq l$, and the proof of both statements is very similar to that of Lemma \ref{result}.  

\begin{lem}
The involution on $B(C_n,R,\delta)$ corresponds to the standard involution on $V(n,k)\otimes V(n,k) \otimes B_{n-2k}$ given by $:e \otimes f \otimes \pi \mapsto f \otimes e \otimes \pi^{-1}$.
\end{lem}
\begin{proof}
This follows easily from the definition of the involution as the reflection in the horizontal axis.
\end{proof}

\begin{thm}
Assume that $\delta \neq0$.  The Brauer algebra $B(C_n,R, \delta)$ is a cellularly stratified diagram algebra.
\end{thm}
\begin{proof}
By the above we have that the algebra is an iterated inflation of hyperoctahedral groups.  Therefore we need only check that the conditions on the idempotents are satisfied.  But clearly the elements $f_k$ that we have defined satisfy the required conditions.
\end{proof}

The following corollaries are immediate from the theory of cellularly stratified algebras, and can all be found in this general setting in \cite{damn}.  

\begin{cor}
We have that the cell-modules of  $B(C_n,R, \delta)$ are 
\begin{align*}
\{ \Delta (\mu,k) = V(n,k) \otimes v_k \otimes \Delta(\mu) : k \leq n/2,  \Delta(\mu) \text{ is a standard module of } H_{n-2k}\}
\end{align*}
for $v_k$ is an arbitrary non-zero element of $V(n,k)$.
\end{cor}

We have that the simple modules of $H_{n-2k}$ are $D(\mu)=\Delta(\mu)/\rad \Delta(\mu)$ and that the simple modules of $B(C_n,R, \delta)$ are $D(\mu,k)=\Delta(\mu,k)/\rad \Delta(\mu,k)$. 

\begin{cor}
Assume that $R$ is a field of characteristic $p\geq 0$. The Brauer algebra $B(C_n,R, \delta)$ is quasi-hereditary if and only if $\delta \neq 0$ and $p > n$.
\end{cor}

\begin{proof}
A chain of cell-ideals in the input algebras is always induced to a chain of cell-ideals in the iterated inflation.  The chain is quasi-hereditary precisely when all the input algebras are quasi-hereditary, and there is no layer whose bilinear form is zero.  

We have that the group algebras $RH_{n-2k}$ are quasi-hereditary if and only if $p > n$.  There is a layer of the iterated inflation such that $\varphi(-,-)=0$ if and only if $\delta = 0$.  Therefore we are done.
\end{proof}

\begin{cor}
The decomposition matrices of the hyperoctahedral groups appear along the diagonal of the decomposition matrix for the Brauer algebra of type $C$.  More specifically,
\begin{align*}
[\Delta(\mu,k):D(\lambda,k)]=[\Delta(\mu):D(\lambda)].  
\end{align*}
\end{cor}

\end{document}